\documentclass[11pt]{amsart}
\usepackage{amssymb}
\usepackage{t1enc}
\usepackage[latin2]{inputenc}
\usepackage{verbatim}
\usepackage{amsmath,amsfonts,amssymb,amsthm}
\usepackage[mathcal]{eucal}
\usepackage{enumerate}
\usepackage[centertags]{amsmath}
\usepackage{graphics}

\newtheorem{theorem}{Theorem}
\newtheorem{lemma}{Lemma}

\newtheorem{corollary}{Corollary}

\newtheorem{remark}{Remark}
\numberwithin{equation}{subsection}

\begin{document}
\author{George Tephnadze}
\title[Walsh-Kaczmarz-Nörlund means]{On the maximal operators of
Walsh-Kaczmarz-Nörlund means}
\address{G. Tephnadze, Department of Mathematics, Faculty of Exact and
Natural Sciences, Tbilisi State University, Chavchavadze str. 1, Tbilisi
0128, Georgia and Department of Engineering Sciences and Mathematics, Lule%
\aa {} University of Technology, SE-971 87, Lule\aa {}, Sweden.}
\email{giorgitephnadze@gmail.com}
\thanks{The research was supported by Shota Rustaveli National Science
Foundation grant no.13/06 (Geometry of function spaces, interpolation and
embedding theorems).}
\date{}
\maketitle

\begin{abstract}
The main aim of this paper is to investigate $\left( H_{p},L_{p,\infty
}\right) $ type inequalities for maximal operators of Nörlund means with
monotone coefficients of one-dimensional Walsh-Kaczmarz system. By applying
this results we conclude a.e convergence of such Walsh-Kaczmarz-Nörlund
means.
\end{abstract}




\section{Introduction}

In 1948 $\breve{\text{S}}$neider \cite{sne} introduced the Walsh-Kaczmarz
system and showed that the inequality $\limsup_{n\rightarrow \infty
}D_{n}^{\kappa }(x)/\log n\geq C>0$ holds a.e. In 1974 Schipp \cite{Sch2}
and Young \cite{Y} proved that the Walsh-Kaczmarz system is a convergence
system. Skvortsov \cite{Sk1} in 1981 showed that the Fejér means with
respect to the Walsh-Kaczmarz system converge uniformly to $f$ for any
continuous functions $f$. Gát \cite{gat} proved that, for any integrable
functions, the Fejér means with respect to the Walsh-Kaczmarz system
converges almost everywhere to the function. He showed that the maximal
operator $\sigma ^{\ast ,\kappa }$ of \ Walsh-Kaczmarz-Fejér means is of
weak type $(1,1)$ and of type $(p,p)$ for all $1<p\leq \infty $. Gát's
result was generalized by Simon \cite{S2}, who showed that the maximal
operator $\sigma ^{\ast ,\kappa }$ is of type $(H_{p},L_{p})$ for $p>1/2$.
In the endpoint case $p=1/2$ Goginava \cite{Gog-PM} (see also \cite{GGN}, 
\cite{tep1} and \cite{tep2}) proved that maximal operator $\sigma ^{\ast
,\kappa }$ of Walsh-Kaczmarz-Fejér means is not of type $(H_{1/2},L_{1/2})$
and Weisz \cite{We5} showed that the following is true:

\textbf{Theorem W1. }The\textbf{\ }maximal operator $\sigma ^{\ast ,\kappa }$
of \ Walsh-Kaczmarz-Fejér means is bounded from the Hardy space $H_{1/2}$ to
the space $L_{1/2,\infty }$.

The almost everywhere convergence of $\left( C,\alpha \right) $ $\left(
0<\alpha <1\right) $ means with respect Walsh-Kaczmarz system was considered
by Goginava \cite{gog1}. Gát and Goginava \cite{gago} proved that the
following is true:

\textbf{Theorem G2. }The\textbf{\ }maximal operator $\sigma ^{\alpha ,\ast
,\kappa }$ of $\left( C,\alpha \right) $ $\left( 0<\alpha <1\right) $ means
with respect Walsh-Kaczmarz system is bounded from the Hardy space $%
H_{1/\left( 1+\alpha \right) }$ to the space $L_{1/\left( 1+\alpha \right)
,\infty }$.

Goginava and Nagy \cite{gona} proved that $\sigma ^{\alpha ,\ast ,\kappa }$
is not bounded from the Hardy space $H_{1/\left( 1+\alpha \right) }$ to the
space $L_{1/\left( 1+\alpha \right) }$.

Logarithmic means with respect to the Walsh and Vilenkin systems systems was
studied by several authors. We mention, for instance, the papers by Simon 
\cite{Si1}, Gát \cite{Ga1} and Blahota, Gát \cite{bg}, (see also \cite{tep4}%
). In \cite{G-N-Complex} Goginava and Nagy proved that the maximal operator $%
R^{\ast ,\kappa }$ of Riesz`s means is bounded from the Hardy space $H_{p}$
to the space $weak-L_{p}$, when $p>1/2,$ but is not bounded from the Hardy
space $H_{p}$ to the space $L_{p},$ when $0<p\leq 1/2.$ They also showed
that there exists a martingale $f\in H_{p},$ $(0<p\leq 1),$ such that the
maximal operator $L^{\ast ,\kappa }$ of Nörlund logarithmic means is not
bounded in the space $L_{p}.$

In the two dimensional case approximation properties of Nörlund and Cesáro
means was considered by Nagy (see \cite{na}, \cite{n} and \cite{nagy}). The
results for summability of some Nörlund means of Walsh-Fourier series can be
found in \cite{gog8} and \cite{tep3}.

The main aim of this paper is to investigate $(H_{p},L_{p,\infty })$-type
inequalities for the maximal operators of Nörlund means with monotone
coefficients of one-dimensional Kaczmarz-Fourier series.

This paper is organized as follows: in order not to disturb our discussions
later on some definitions and notations are presented in Section 2. The main
results and some of its consequences can be found in Section 3. For the
proofs of the main results we need some auxiliary results of independent
interest. Also these results are presented in Section 3. The detailed proofs
are given in Section 4.

\section{Definitions and Notations}

Now, we give a brief introduction to the theory of dyadic analysis \cite%
{SWSP}. Let $\mathbf{N}_{+}$ denote the set of positive integers, $\mathbf{%
N:=N}_{+}\cup \{0\}.$

Denote ${\mathbb{Z}}_{2}$ the discrete cyclic group of order 2, that is ${%
\mathbb{Z}}_{2}=\{0,1\},$ where the group operation is the modulo 2 addition
and every subset is open. The Haar measure on ${\mathbb{Z}}_{2}$ is given
such that the measure of a singleton is 1/2. Let $G$ be the complete direct
product of the countable infinite copies of the compact groups ${\mathbb{Z}}%
_{2}.$ The elements of $G$ are of the form 
\begin{equation*}
x=\left( x_{0},x_{1},...,x_{k},...\right) ,\text{ \ \ \ }x_{k}=0\vee 1,\text{
\ }\left( k\in \mathbf{N}\right) .
\end{equation*}

The group operation on $G$ is the coordinate-wise addition, the measure
(denoted by $\mu $) and the topology are the product measure and topology.
The compact Abelian group $G$ is called the Walsh group. A base for the
neighborhoods of $G$ can be given in the following way: 
\begin{equation*}
I_{0}\left( x\right) :=G,\quad I_{n}\left( x\right) :=I_{n}\left(
x_{0},...,x_{n-1}\right) :=\left\{ y\in G:\,y=\left(
x_{0},...,x_{n-1},y_{n},y_{n+1},...\right) \right\} ,
\end{equation*}%
$\left( x\in G,n\in \mathbf{N}\right) .$ These sets are called dyadic
intervals. Denote by $0=\left( 0:i\in \mathbf{N}\right) \in G$ the null
element of $G.$ Let $I_{n}:=I_{n}\left( 0\right) ,$ $\overline{I_{n}}%
:=G\backslash I_{n}\,\left( n\in \mathbf{N}\right) .$ Set $e_{n}:=\left(
0,...,0,1,0,...\right) \in G,$ the $n$-th coordinate of which is 1 and the
rest are zeros $\left( n\in \mathbf{N}\right) .$

For $k\in \mathbf{N}$ and $x\in G$ let us denote the $k$-th Rademacher
function, by 
\begin{equation*}
r_{k}\left( x\right) :=\left( -1\right) ^{x_{k}}.
\end{equation*}

Now, define the Walsh system $w:=(w_{n}:n\in \mathbf{N})$ on $G$ as: 
\begin{equation*}
w_{n}(x):=\overset{\infty }{\underset{k=0}{\Pi }}r_{k}^{n_{k}}\left(
x\right) =r_{\left\vert n\right\vert }\left( x\right) \left( -1\right) ^{%
\underset{k=0}{\overset{\left\vert n\right\vert -1}{\sum }}n_{k}x_{k}}\text{%
\qquad }\left( n\in \mathbf{N}\right) .
\end{equation*}

If $n\in \mathbf{N}$, then $n=\sum\limits_{i=0}^{\infty }n_{i}2^{i}$ can be
written, where $n_{i}\in \{0,1\}\quad \left( i\in \mathbf{N}\right) $, i. e. 
$n$ is expressed in the number system of base 2.

Denote $\left\vert n\right\vert :=\max \{j\in \mathbf{N;}n_{j}\neq 0\},$
that is $2^{\left\vert n\right\vert }\leq n<2^{\left\vert n\right\vert +1}.$

The Walsh-Kaczmarz functions are defined by

\begin{equation*}
\kappa _{n}\left( x\right) :=r_{\left\vert n\right\vert }\left( x\right)
\prod\limits_{k=0}^{\left\vert n\right\vert -1}\left( r_{\left\vert
n\right\vert -1-k}\left( x\right) \right) ^{n_{k}}=r_{\left\vert
n\right\vert }\left( x\right) \left( -1\right)
^{\sum\limits_{k=0}^{\left\vert n\right\vert -1}n_{k}x_{_{\left\vert
n\right\vert -1-k}}}.
\end{equation*}

The Dirichlet kernels are defined

\begin{equation*}
D_{0}:=0\text{, \ \ \ \ }D_{n}^{\psi }:=\sum_{i=0\text{ }}^{n-1}\psi _{i},%
\text{ }\left( \psi =w,\text{ or }\psi =\kappa \right) .
\end{equation*}

The $2^{n}$-th Dirichlet kernels have a closed form (see e.g. \cite{SWSP})%
\begin{equation}
D_{2^{n}}^{w}\left( x\right) =D_{2^{n}}\left( x\right) =D_{2^{n}}^{\kappa
}\left( x\right) \text{\thinspace }=\left\{ 
\begin{array}{ll}
2^{n}\text{ \ \ \ \ \ } & x\in I_{n}, \\ 
0 & x\notin I_{n}.%
\end{array}%
\right.  \label{Dir}
\end{equation}

The norm (or quasi-norm) of the spaces $L_{p}(G)$ and $L_{p,\infty }\left(
G\right) $ are defined by \qquad

\begin{equation*}
\left\Vert f\right\Vert _{p}^{p}:=\int_{G}\left\vert f\right\vert ^{p}d\mu ,%
\text{ \ \ }\left\Vert f\right\Vert _{L_{p,\infty }(G)}^{p}:=\underset{%
\lambda >0}{\sup }\lambda ^{p}\mu \left( f>\lambda \right) ,\text{ \ \ \ }%
\left( 0<p<\infty \right) ,
\end{equation*}%
respectively.

The $\sigma $-algebra generated by the dyadic intervals of measure $2^{-k}$
will be denoted by $F_{k}$ $\left( k\in \mathbf{N}\right) .$ Denote by $%
f=\left( f^{\left( n\right) },n\in \mathbf{N}\right) $ a martingale with
respect to $\left( F_{n},n\in \mathbf{N}\right) $ (for details see, e. g. 
\cite{we2,we3}). The maximal function of a martingale $f$ \ is defined by 
\begin{equation*}
f^{\ast }=\sup\limits_{n\in \mathbf{N}}\left\vert f^{\left( n\right)
}\right\vert .
\end{equation*}

In case $f\in L_{1}\left( G\right) $, the maximal function can also be given
by 
\begin{equation*}
f^{\ast }\left( x\right) =\sup\limits_{n\in \mathbf{N}}\frac{1}{\mu \left(
I_{n}(x)\right) }\left\vert \int\limits_{I_{n}(x)}f\left( u\right) d\mu
\left( u\right) \right\vert ,\ \ x\in G.
\end{equation*}

For $0<p<\infty $ the Hardy martingale space $H_{p}(G)$ consists of all
martingales for which

\begin{equation*}
\left\| f\right\| _{H_{p}}:=\left\| f^{*}\right\| _{p}<\infty .
\end{equation*}

If $f\in L_{1}\left( G\right) ,$ then it is easy to show that the sequence $%
\left( S_{2^{n}}f:n\in \mathbf{N}\right) $ is a martingale.

If $\ f$ \ is a martingale, then the Walsh-Kaczmarz-Fourier coefficients
must be defined in a little bit different way: 
\begin{equation*}
\widehat{f}^{\psi }\left( i\right) =\lim\limits_{n\rightarrow \infty
}\int\limits_{G}f^{\left( n\right) }\psi _{i}d\mu ,\text{ \ \ }\left( \psi
=w,\text{ or }\psi =\kappa \right) .
\end{equation*}

The Walsh-Kaczmarz-Fourier coefficients of $f\in L_{1}\left( G\right) $ are
the same as the ones of the martingale $\left( S_{2^{n}}f:n\in \mathbf{N}%
\right) $ obtained from $f$.

The partial sums of the Walsh-Kaczmarz-Fourier series are defined as
follows: 
\begin{equation*}
S_{M}^{\psi }f:=\sum\limits_{i=0}^{M-1}\widehat{f}\left( i\right) \psi _{i},%
\text{ \ \ }\left( \psi =w,\text{ or }\psi =\kappa \right) .
\end{equation*}

Let \{$q_{k}:$ $k>0$\} be a sequence of nonnegative numbers. The $n$-th Nö%
rlund means for the Fourier series of $\ f$ \ is defined by

\begin{equation}
t_{n}^{\psi }:=\frac{1}{Q_{n}}\overset{n}{\underset{k=1}{\sum }}%
q_{n-k}S_{k}^{\psi }f,\text{ \ \ }\left( \psi =w,\text{ or }\psi =\kappa
\right) ,  \label{nor}
\end{equation}%
where 
\begin{equation*}
Q_{n}:=\sum_{k=0}^{n-1}q_{k}.
\end{equation*}

It is evident that 
\begin{equation*}
t_{n}^{\psi }f\left( x\right) =\int_{G}f\left( x+t\right) F_{n}^{\psi
}\left( t\right) dt,\text{ }
\end{equation*}%
where%
\begin{equation*}
\text{\ }F_{n}^{\psi }=\frac{1}{Q_{n}}\overset{n}{\underset{k=1}{\sum }}%
q_{n-k}D_{k}^{\psi }.
\end{equation*}

Let $q_{0}>0$ and $\underset{n\rightarrow \infty }{\lim }Q_{n}=\infty .$ The
summability method (\ref{nor}) generated by $\{q_{k}:k\geq 0\}$ is regular
if and only if

\begin{equation}
\underset{n\rightarrow \infty }{\lim }\frac{q_{n-1}}{Q_{n}}=0.  \label{1a11}
\end{equation}

It can be found in \cite{mor} (see also \cite{moore}).

The $n$-th Fejér means of a function $f$ \ is given by

\begin{equation*}
\sigma _{n}^{\psi }f:=\frac{1}{n}\sum_{k=0}^{n-1}S_{k}^{\psi }f,\text{ \ \ \ 
}\left( \psi =w,\text{ or }\psi =\kappa \right) .
\end{equation*}

\bigskip Fejér kernel is defined in the usual manner%
\begin{equation*}
K_{n}^{\psi }:=\frac{1}{n}\overset{n}{\underset{k=1}{\sum }}D_{k}^{\psi },%
\text{ \ \ \ }\left( \psi =w,\text{ or }\psi =\kappa \right) .
\end{equation*}

The $\left( C,\alpha \right) $-means are defined as

\begin{equation*}
\sigma _{n}^{\alpha ,\psi }f=\frac{1}{A_{n}^{\alpha }}\overset{n}{\underset{%
k=1}{\sum }}A_{n-k}^{\alpha -1}S_{k}^{\psi }f,\text{ \ \ \ \ }\left( \psi =w,%
\text{ or }\psi =\kappa \right) ,
\end{equation*}%
where

\begin{equation}
A_{0}^{\alpha }=0,\text{ \qquad }A_{n}^{\alpha }=\frac{\left( \alpha
+1\right) ...\left( \alpha +n\right) }{n!},~~\alpha \neq -1,-2,...
\label{1d}
\end{equation}

It is known that%
\begin{equation}
A_{n}^{\alpha }\sim n^{\alpha },\text{ \ }A_{n}^{\alpha }-A_{n-1}^{\alpha
}=A_{n}^{\alpha -1},\text{ }\overset{n}{\underset{k=1}{\sum }}%
A_{n-k}^{\alpha -1}=A_{n}^{\alpha }.  \label{2d}
\end{equation}

The kernel of $\left( C,\alpha \right) $-means is defined in the following
way 
\begin{equation*}
K_{n}^{\alpha ,\psi }f=\frac{1}{A_{n}^{\alpha }}\overset{n}{\underset{k=1}{%
\sum }}A_{n-k}^{\alpha -1}D_{k}^{\psi }f,\text{ \ \ \ \ }\left( \psi =w,%
\text{ or }\psi =\kappa \right) .
\end{equation*}

The $n$-th Riesz`s logarithmic mean $R_{n}$ and Nörlund logarithmic mean $%
L_{n}$ are defined by

\begin{equation*}
R_{n}^{\psi }f:=\frac{1}{l_{n}}\overset{n-1}{\underset{k=0}{\sum }}\frac{%
S_{k}^{\psi }f}{k},\text{ \ }L_{n}^{\psi }f:=\frac{1}{l_{n}}\overset{n-1}{%
\underset{k=1}{\sum }}\frac{S_{k}^{\psi }f}{n-k},\text{ \ \ }\left( \psi =w,%
\text{ or }\psi =\kappa \right) .
\end{equation*}%
respectively, where%
\begin{equation*}
l_{n}:=\sum_{k=1}^{n-1}1/k.
\end{equation*}

For the martingale $f$ we consider the following maximal operators%
\begin{equation*}
\text{\ }t^{\ast ,\psi }f:=\sup_{n\in 
\mathbb{N}
}\left\vert t_{n}^{\psi }f\right\vert ,\text{ \ }\sigma ^{\ast ,\psi
}f:=\sup_{n\in 
\mathbb{N}
}\left\vert \sigma _{n}^{\psi }f\right\vert ,\text{ \ \ }\sigma ^{\alpha
,\ast ,\psi }f:=\sup_{n\in 
\mathbb{N}
}\left\vert \sigma _{n}^{\alpha ,\psi }f\right\vert ,
\end{equation*}%
\begin{equation*}
R^{\ast ,\psi }f:=\sup_{n\in 
\mathbb{N}
}\left\vert R_{n}^{\psi }f\right\vert ,\text{\ \ \ }L^{\ast ,\psi
}f:=\sup_{n\in 
\mathbb{N}
}\left\vert L_{n}^{\psi }f\right\vert ,\text{ \ \ }(\psi =w,\text{ or \ }%
\psi =\kappa ).
\end{equation*}

A bounded measurable function $a$ is p-atom, if there exists an interval $I$%
, such that

\begin{equation*}
\int_{I}ad\mu =0,\text{ \ \ }\left\Vert a\right\Vert _{\infty }\leq \mu
\left( I\right) ^{-1/p},\text{ \ \ supp}\left( a\right) \subset I.
\end{equation*}

\section{Results}

\begin{center}
\textbf{Main results and some of its consequences}
\end{center}

\begin{theorem}
\label{Theorem2}a) Let sequence $\{q_{k}:k\geq 0\}$ be non-increasing,
satisfying condition 
\begin{equation}
\frac{q_{0}n}{Q_{n}}=O\left( 1\right) ,\text{ \ \ \ \ as \ }n\rightarrow
\infty ,  \label{100}
\end{equation}%
or non-decreasing. Then the maximal operators $t^{\ast ,\kappa }$ of Nörlund
means are bounded from the Hardy space $H_{1/2}$ to the space $L_{1/2,\infty
}.$

b)\textit{\ Let} $0<p<1/2$ and \textit{sequence }$\{q_{k}:k\geq 0\}$ \textit{%
be non-decreasing sequence}, \textit{satisfying condition} 
\begin{equation}
\frac{q_{0}}{Q_{n}}\geq \frac{1}{n},  \label{cond0}
\end{equation}%
or \textit{non-increasing. Then there exists a martingale }$f\in H_{p}\left(
G\right) ,$ such that\textit{\ } 
\begin{equation*}
\underset{n\in 
\mathbb{N}
}{\sup }\frac{\left\Vert t_{n}^{\kappa }f\right\Vert _{L_{p,\infty }}}{%
\left\Vert f\right\Vert _{H_{p}}}=\infty .
\end{equation*}
\end{theorem}

\begin{theorem}
\label{Theorem3}a) Let $0<\alpha <1$, sequence $\{q_{k}:k\geq 0\}$ be
non-increasing and%
\begin{equation}
\frac{q_{0}n^{\alpha }}{Q_{n}}=O\left( 1\right) ,\text{ \ \ }\frac{%
q_{n}-q_{n+1}}{n^{\alpha -2}}=O\left( 1\right) ,\text{ \ \ \ \ as \ }%
n\rightarrow \infty .  \label{no1}
\end{equation}%
\textit{Then the maximal operator }$t^{\ast ,\kappa }$\textit{\ of Nörlund
means are bounded from the Hardy space }$H_{1/\left( 1+\alpha \right) }$%
\textit{\ to the space }$L_{1/\left( 1+\alpha \right) ,\infty }.$

\textit{b) Let }$0<p<1/\left( 1+\alpha \right) $\textit{, sequence }$%
\{q_{k}:k\geq 0\}$\textit{\ be non-increasing and}%
\begin{equation}
\frac{q_{0}}{Q_{n}}\geq \frac{c}{n^{\alpha }},\text{ }0<\alpha \leq 1,\text{
\ as \ }n\rightarrow \infty .  \label{nom3}
\end{equation}%
\textit{Then there exists an martingale }$f\in H_{p}\left( G\right) ,$ 
\textit{such that}%
\begin{equation*}
\underset{n\in 
\mathbb{N}
}{\sup }\frac{\left\Vert t_{n}^{\kappa }f\right\Vert _{L_{p,\infty }}}{%
\left\Vert f\right\Vert _{H_{p}}}=\infty .
\end{equation*}

\textit{c) Let sequence }$\{q_{k}:k\geq 0\}$\textit{\ be non-increasing and }
\begin{equation}
\overline{\lim_{n\rightarrow \infty }}\frac{q_{0}n^{\alpha }}{Q_{n}}=\infty ,
\label{nom2}
\end{equation}%
\textit{Then there exists a martingale }$f\in H_{p}\left( G\right) ,$ 
\textit{such that}%
\begin{equation*}
\underset{n\in 
\mathbb{N}
}{\sup }\frac{\left\Vert t_{n}^{\kappa }f\right\Vert _{L_{1/\left( 1+\alpha
\right) ,\infty }}}{\left\Vert f\right\Vert _{H_{1/\left( 1+\alpha \right) }}%
}=\infty .
\end{equation*}%
The next remark shows that conditions in (\ref{no1}) are sharp in the
following sense:
\end{theorem}

\begin{remark}
The sequence $\{q_{k}:k\geq 0\}$ of Cesáro means $\sigma _{n}^{\alpha }$
satisfy conditions%
\begin{equation}
\frac{q_{0}}{Q_{n}}\geq \frac{c}{n^{\alpha }},\text{ \ \ \ }%
q_{n}-q_{n+1}\geq \frac{c}{n^{\alpha -2}},\text{ \ \ }0<\alpha \leq 1,\text{
\ as \ }n\rightarrow \infty ,  \label{cond5}
\end{equation}%
but they are not uniformly bounded from the martingale Hardy spaces $%
H_{1/\left( 1+\alpha \right) }\left( G\right) $ to the space $L_{1/\left(
1+\alpha \right) }\left( G\right) .$
\end{remark}

Theorem \ref{Theorem2} follows the following result:

\begin{corollary}
\label{Corollary1}Let $\left\{ q_{k}=\log ^{\left( \beta \right) }\left(
k+1\right) ^{\alpha }:k\geq 0\right\} ,$ where $\alpha \geq 0$, $\beta \in 
\mathbb{N}
_{+}$ and $\log ^{\left( \beta \right) }x=\overset{\beta \text{ times}}{%
\overbrace{\log ...\log x}}.$ Then the following summability method%
\begin{equation*}
\theta _{n}^{\kappa }f=\frac{1}{Q_{n}}\overset{n}{\underset{k=1}{\sum }}\log
^{\left( \beta \right) }\left( n-k-1\right) ^{\alpha }S_{k}^{\kappa }f
\end{equation*}%
is bounded from the Hardy space $H_{1/2}$ to the space $weak-L_{1/2}$ and is
not bounded from $H_{p}$ to the space $weak-L_{p},$ when $0<p<1/2.$
\end{corollary}

Analogously to Theorem \ref{Theorem2}, if we apply Abel transformation we
obtain that the following is true:

\begin{corollary}
\label{Corollary2}The maximal operator $R^{\ast ,\kappa }$ of \ Riesz`s
means is bounded from the Hardy space $H_{1/2}$ to the space $weak-L_{1/2}$
and is not bounded from $H_{p}$ to the space $weak-L_{p},$ when $0<p<1/2.$
\end{corollary}

By combining first and second part of Theorem \ref{Theorem3} we prove that
the following is true:

\begin{corollary}
\label{Corollary3}Let $\left\{ q_{k}=k^{\alpha -1}:k\geq 0\right\} ,$ where $%
0<\alpha \leq 1.$ Then the following summability method%
\begin{equation*}
L_{n}^{\alpha ,\kappa }f=\frac{1}{Q_{n}}\overset{n}{\underset{k=1}{\sum }}%
\left( n-k\right) ^{\alpha -1}S_{k}^{\kappa }f
\end{equation*}%
is bounded from the Hardy space $H_{1/\left( 1+\alpha \right) }$ to the
space $weak-L_{1/\left( 1+\alpha \right) }$ and is not bounded from $H_{p}$
to the space $weak-L_{p},$ when $0<p<1/\left( 1+\alpha \right) .$
\end{corollary}

By applying second part of Theorem \ref{Theorem3} we obtain that the
following is true

\begin{corollary}
\label{Corollary4}The maximal operator $L^{\ast ,\kappa }$ of \ Nörlund
logarithmic means is not bounded from the Hardy space $H_{p}$ to the space $%
weak-L_{p}$, when $0<p<1.$
\end{corollary}

By using Lemma \ref{Lemma1a} we get that

\begin{corollary}
\label{Corollary5}Let $f\in L_{1}$ and $\{q_{k}:k\geq 0\}$ be non-decreasing
or \textit{non-increasing } satisfying condition (\ref{no1}). Then $%
t_{n}^{\kappa }f\rightarrow f,$ a.e.
\end{corollary}

As the consequence of corollaries \ref{Corollary2} and \ref{Corollary5} we
conclude that

\begin{corollary}
\label{Corollary6}Let $f\in L_{1}$. Then%
\begin{equation*}
\sigma _{n}^{\kappa }f\rightarrow f,\text{ \ \ \ a.e., \ \ \ as \ }%
n\rightarrow \infty ,
\end{equation*}%
\begin{equation*}
R_{n}^{\kappa }f\rightarrow f,\text{ \ \ \ a.e., \ \ \ \ as \ }n\rightarrow
\infty
\end{equation*}%
and%
\begin{equation*}
\sigma _{n}^{\alpha ,\kappa }f\rightarrow f,\text{ \ \ \ a.e., \ \ as \ }%
n\rightarrow \infty ,\text{ \ \ }\left( 0<\alpha <1\right) .
\end{equation*}
\end{corollary}

\bigskip

\begin{center}
\textbf{Some auxiliary results}
\end{center}

\begin{lemma}
\label{Lemma1a}(see \cite{we2}) Suppose that an operator $T$ is $\sigma $%
-linear and for some $0<p\leq 1$
\end{lemma}

\begin{equation*}
\int\limits_{\overline{I}}\left\vert Ta\right\vert ^{p}d\mu \leq
c_{p}<\infty ,
\end{equation*}%
for every $p$-atom $a$, where $I$ denote the support of the atom. If $T$ is
bounded from $L_{\infty \text{ }}$ to $L_{\infty },$ then 
\begin{equation*}
\left\Vert Tf\right\Vert _{L_{p}\left( G\right) }\leq c_{p}\left\Vert
f\right\Vert _{H_{p}\left( G\right) }.
\end{equation*}

Moreover, if $0<p<1$ then $T$ \ is of weak type-(1,1):%
\begin{equation*}
\left\Vert Tf\right\Vert _{L_{1,\infty }\left( G\right) }\leq c\left\Vert
f\right\Vert _{L_{1}\left( G\right) }.
\end{equation*}

\begin{lemma}
\label{Lemma2}\bigskip Let $2^{m}<n\leq 2^{m+1}.$ Then%
\begin{eqnarray*}
Q_{n}F_{n}^{w} &=&Q_{n}D_{2^{m}}-w_{2^{m-1}}\overset{2^{m}-1}{\underset{l=1}{%
\sum }}\left( q_{n-2^{m}+l}-q_{n-2^{m}+l+1}\right) lK_{l}^{w} \\
&&-w_{2^{m}-1}\left( 2^{m}-1\right)
q_{n-1}K_{2^{m}-1}^{w}+w_{2^{m}}Q_{n-2^{m}}F_{n-2^{m}}^{w}.
\end{eqnarray*}
\end{lemma}

\begin{lemma}
\label{Lemma3}Let $0<\alpha <1$ and sequence $\{q_{k}:k\geq 0\}$ be
non-increasing and satisfying condition (\ref{no1}). Then%
\begin{equation*}
\left\vert F_{n}^{w}\right\vert \leq \frac{c\left( \alpha \right) }{%
n^{\alpha }}\overset{\left\vert n\right\vert }{\underset{j=0}{\sum }}%
2^{j\alpha }K_{2^{j}}^{w}.
\end{equation*}
\end{lemma}

\section{Proofs}

\begin{proof}[Proof of Lemma \protect\ref{Lemma2}]
Let $2^{m}<n\leq 2^{m+1}.$ It is easy to show that 
\begin{equation}
\overset{n}{\underset{l=1}{\sum }}q_{n-l}D_{l}^{w}=\overset{2^{m}}{\underset{%
l=1}{\sum }}q_{n-l}D_{l}^{w}+\overset{n}{\underset{l=2^{m}+1}{\sum }}%
q_{n-l}D_{l}^{w}=I+II.  \label{nor1}
\end{equation}

By combining Abel transformation and following equality (See \cite{gog9})%
\begin{equation*}
D_{2^{m}-j}=D_{2^{m}}-w_{2^{m}-1}D_{j},\text{ \ }j=1,...,2^{m}-1,
\end{equation*}%
we get that%
\begin{eqnarray}
I &=&\overset{2^{m}-1}{\underset{l=0}{\sum }}q_{n-2^{m}+l}D_{2^{m}-l}^{w}=%
\overset{2^{m}-1}{\underset{l=1}{\sum }}%
q_{n-2^{m}+l}D_{2^{m}-l}^{w}+q_{n-2^{m}}D_{2^{m}}  \label{nor2} \\
&=&D_{2^{m}}\overset{2^{m}-1}{\underset{l=0}{\sum }}q_{n-2^{m}+l}-w_{2^{m}-1}%
\overset{2^{m}-1}{\underset{l=1}{\sum }}q_{n-2^{m}+l}D_{l}^{w}  \notag \\
&=&\left( Q_{n}-Q_{n-2^{m}}\right) D_{2^{m}}-w_{2^{m}-1}\overset{2^{m}-2}{%
\underset{l=1}{\sum }}\left( q_{n-2^{m}+l}-q_{n-2^{m}+l+1}\right) lK_{l}^{w}
\notag \\
&&-w_{2^{m}-1}q_{n-1}\left( 2^{m}-1\right) K_{2^{m}-1}^{w}  \notag
\end{eqnarray}

Since 
\begin{equation*}
D_{j+2^{m}}^{w}=D_{2^{m}}+w_{2^{m}}D_{j}^{w},\text{ \ \ }j=1,2,...,2^{m}-1
\end{equation*}%
for $II$ we can write that%
\begin{equation}
II=\overset{n-2^{m}}{\underset{l=1}{\sum }}%
q_{n-2^{m}-l}D_{l+2^{m}}^{w}=Q_{n-2^{m}}D_{2^{m}}+w_{2^{m}}Q_{n-2^{m}}F_{n-2^{m}}^{w}
\label{nor3}
\end{equation}

Combining (\ref{nor1}-\ref{nor3}) we complete the proof of Lemma \ref{Lemma2}%
.
\end{proof}

\bigskip

\begin{proof}[Proof of Lemma \protect\ref{Lemma3}]
Let sequence $\{q_{k}:k\geq 0\}$ be non-increasing. The case $q_{0}n/Q_{n}=%
\overset{\_}{O}\left( 1\right) ,$\ as \ $n\rightarrow \infty ,$ will be
considered separately in Theorem \ref{Theorem2}. So, we can exclude this
case.

Since $0<\alpha <1,$ we may assume that $\{q_{k}:k\geq 0\}$ satisfy both
conditions in (\ref{no1}) and in addition, satisfies the following 
\begin{equation*}
\frac{Q_{n}}{q_{0}n}=\overset{\_}{o}\left( 1\right) ,\text{ \ \ \ \ as \ }%
n\rightarrow \infty .
\end{equation*}

It follows that 
\begin{equation}
q_{n}=q_{0}\frac{q_{n}n}{q_{0}n}\leq q_{0}\frac{Q_{n}}{q_{0}n}=\overset{\_}{o%
}\left( 1\right) ,\text{ \ \ \ \ as \ }n\rightarrow \infty  \label{102}
\end{equation}%
By using (\ref{102}) we immediately get that 
\begin{equation}
q_{n}=\overset{\infty }{\underset{l=n}{\sum }}\left( q_{l}-q_{l+1}\right)
\leq \overset{\infty }{\underset{l=n}{\sum }}\frac{1}{l^{2-\alpha }}\leq 
\frac{c}{n^{1-\alpha }}  \label{103}
\end{equation}%
and%
\begin{equation}
Q_{n}\leq \overset{n-1}{\underset{l=0}{\sum }}q_{l}\leq \overset{n}{\underset%
{l=1}{\sum }}\frac{c}{l^{1-\alpha }}\leq cn^{\alpha }  \label{104}
\end{equation}

If we apply (\ref{103}) and (\ref{104}) we get that 
\begin{equation}
Q_{n}D_{2^{m}}\leq 2^{\alpha \left( m+1\right) }D_{2^{m}}\leq
cA_{2^{m}}^{\alpha }D_{2^{m}},\text{ \ }2^{m}<n\leq 2^{m+1}  \label{a1}
\end{equation}%
and 
\begin{equation}
\left( 2^{m}-1\right) q_{n-1}\left\vert K_{2^{m}-1}^{w}\right\vert \leq
cn^{\alpha -1}2^{m}\left\vert K_{2^{m}-1}^{w}\right\vert \leq cA_{n}^{\alpha
-1}2^{m}\left\vert K_{2^{m}-1}^{w}\right\vert .  \label{a2}
\end{equation}%
where $A_{n}^{\alpha }$ is defined by (\ref{1d}).

Let 
\begin{equation*}
n=2^{n_{1}}+2^{n_{2}}+...+2^{n_{r}},\text{ \ \ }n_{1}>n_{2}>...>n_{r},\text{
\ }n^{\left( k\right) }=2^{n_{k+1}}+...+2^{n_{r}}
\end{equation*}

By combining (\ref{a1}) and (\ref{a2}) we have that%
\begin{equation*}
\left\vert Q_{n}F_{n}^{w}\right\vert \leq cA_{n^{\left( 0\right) }}^{\alpha
}D_{2^{n_{1}}}+c\overset{2^{n_{1}}-1}{\underset{l=1}{\sum }}\left\vert
A_{n^{\left( 1\right) }+l}^{\alpha -2}\right\vert \left\vert
lK_{l}^{w}\right\vert +cA_{n^{\left( 0\right) }}^{\alpha
-1}2^{n_{1}}\left\vert K_{2^{n_{1}}-1}^{w}\right\vert +c\left\vert
Q_{n^{\left( 1\right) }}F_{n^{\left( 1\right) }}^{w}\right\vert .
\end{equation*}

By using this process $r$-time we get that%
\begin{equation*}
\left\vert Q_{n}F_{n}^{w}\right\vert \leq c\overset{r}{\underset{k=1}{\sum }}%
\left( A_{n^{\left( k-1\right) }}^{\alpha }D_{2^{n_{k}}}+\overset{2^{n_{k}}-1%
}{\underset{l=1}{\sum }}\left\vert A_{n^{\left( k\right) }+l}^{\alpha
-2}\right\vert \left\vert lK_{l}^{w}\right\vert +A_{n^{\left( k-1\right)
}}^{\alpha -1}2^{n_{k}}\left\vert K_{2^{n_{k}}-1}^{w}\right\vert \right) .
\end{equation*}

The next steps of the proof is analogously to Lemma 5 of the paper \cite%
{gago}, where is proved the analogical estimation for $\left( C,\alpha
\right) $ means.
\end{proof}

\bigskip

\begin{proof}[\textbf{Proof of Theorem \protect\ref{Theorem2}}]
By using Abel transformation we obtain that 
\begin{equation}
Q_{n}:=\overset{n-1}{\underset{j=0}{\sum }}q_{j}=\overset{n}{\underset{j=1}{%
\sum }}q_{n-j}\cdot 1=\overset{n-1}{\underset{j=1}{\sum }}\left(
q_{n-j}-q_{n-j-1}\right) j+q_{0}n  \label{2bb}
\end{equation}

and%
\begin{equation}
t_{n}^{\kappa }f=\frac{1}{Q_{n}}\left( \overset{n-1}{\underset{j=1}{\sum }}%
\left( q_{n-j}-q_{n-j-1}\right) j\sigma _{j}^{\kappa }f+q_{0}n\sigma
_{n}^{\kappa }f\right) .  \label{2c}
\end{equation}

Let sequence $\{q_{k}:k\geq 0\}$ be non-increasing, satisfying condition (%
\ref{100}). Then

\begin{eqnarray*}
\left\vert t_{n}^{\kappa }f\right\vert &\leq &\frac{1}{Q_{n}}\left( \overset{%
n-1}{\underset{j=1}{\sum }}\left\vert q_{n-j}-q_{n-j-1}\right\vert
j+q_{0}n\right) \sigma ^{\ast ,\kappa }f \\
&=&\frac{-1}{Q_{n}}\left( \overset{n-1}{\underset{j=1}{\sum }}\left(
q_{n-j}-q_{n-j-1}\right) j\sigma _{j}^{\kappa }f+q_{0}n\sigma _{n}^{\kappa
}f\right) \sigma ^{\ast ,\kappa }f+\frac{2q_{0}n}{Q_{n}}\sigma ^{\ast
,\kappa }f \\
&\leq &c\sigma ^{\ast ,\kappa }f.
\end{eqnarray*}%
Let sequence $\{q_{k}:k\geq 0\}$ be non-decreasing. Then%
\begin{equation*}
\leq \frac{1}{Q_{n}}\left( \overset{n-1}{\underset{j=1}{\sum }}\left(
q_{n-j}-q_{n-j-1}\right) j+q_{0}n\right) \sigma ^{\ast ,\kappa }f\leq
c\sigma ^{\ast ,\kappa }f.
\end{equation*}

It follows that $t^{\ast ,\kappa }f\leq c\sigma ^{\ast ,\kappa }f.$ By using
Theorem W1 we conclude that the maximal operators $t^{\ast ,\kappa }$ are
bounded from the martingale Hardy space $H_{1/2}$ to the space $%
L_{1/2,\infty }.$

It follows that (see Lemma \ref{Lemma1a}) $t^{\ast ,\kappa }$ is of weak
type (1,1) and $t_{n}^{\kappa }f\rightarrow f,$ a.e.

Let%
\begin{equation*}
f_{n}=D_{2^{n+1}}-D_{2^{n}}.
\end{equation*}

It is evident that 
\begin{equation*}
\widehat{f}_{n}^{\kappa }\left( i\right) =\left\{ 
\begin{array}{l}
\text{ }1,\text{ if }i=2^{n},...,2^{n+1}-1, \\ 
\text{ }0,\text{ otherwise.}%
\end{array}%
\right.
\end{equation*}

From (\ref{Dir}) we get that%
\begin{equation}
\left\Vert f_{n}\right\Vert _{H_{p}}=\left\Vert D_{2^{n}}\right\Vert
_{p}\leq 1/2^{n\left( 1/p-1\right) }.  \label{16a}
\end{equation}

It is easy to show that 
\begin{eqnarray*}
\left\vert t_{2^{n}+1}^{\kappa }f_{n}\right\vert &=&\frac{1}{Q_{2^{n}+1}}%
\left\vert q_{0}S_{2^{n}+1}^{\kappa }f_{n}\right\vert =\frac{q_{0}}{%
Q_{2^{n}+1}}\left\vert D_{2^{n}+1}^{\kappa }-D_{2^{n}}\right\vert \\
&=&\frac{q_{0}}{Q_{2^{n}+1}}\left\vert \kappa _{2^{n}}\right\vert =\frac{%
q_{0}}{Q_{2^{n}+1}}.
\end{eqnarray*}

Let sequence $\{q_{k}:k\geq 0\}$ be non-increasing. Then we automatically
get that

\begin{equation}
\frac{q_{0}}{Q_{2^{n}+1}}\geq \frac{q_{0}}{q_{0}\left( 2^{n}+1\right) }=%
\frac{1}{2^{n}+1}.  \label{101}
\end{equation}

Under condition (\ref{cond0}) we also have inequality (\ref{101}) in the
case when sequence $\{q_{k}:k\geq 0\}$ be non-decreasing. Hence

\begin{equation}
\frac{\left\Vert t_{2^{n}+1}^{\kappa }f_{n}\right\Vert _{L_{p,\infty }}}{%
\left\Vert f_{n}\right\Vert _{H_{p}}}\geq \frac{\frac{cq_{0}}{Q_{2^{n}+1}}%
\left( \mu \left\{ x\in G:\left\vert t_{2^{n}+1}^{\kappa }f_{n}\right\vert
\geq \frac{cq_{0}}{Q_{2^{n}+1}}\right\} \right) ^{1/p}}{\left\Vert
f_{n}\right\Vert _{H_{p}}}  \label{17}
\end{equation}%
\begin{equation*}
\geq \frac{cq_{0}2^{n\left( 1/p-1\right) }}{Q_{2^{n}+1}}\geq \frac{%
cq_{0}2^{n\left( 1/p-1\right) }}{2^{n}+1}\geq cq_{0}2^{n\left( 1/p-2\right)
}.
\end{equation*}

Since, $0<p<1/2$ so $n\rightarrow \infty $ gives our statement.
\end{proof}

\bigskip

\begin{proof}[\textbf{Proof of Theorem \protect\ref{Theorem3}}]
Since $t^{\ast ,\kappa }$ is bounded from $L_{\infty }$ to $L_{\infty },$ by
Lemma \ref{Lemma1a}, the proof of theorem \ref{Theorem3} will be complete,
if we show that%
\begin{equation*}
\int\limits_{\overline{I}_{N}}\left\vert t^{\ast ,\kappa }a\right\vert
^{1/\left( 1+\alpha \right) }d\mu \leq c<\infty ,
\end{equation*}%
for every$1/\left( 1+\alpha \right) $-atom $a,$ where $I$ denotes the
support of the atom$.$

To show boundednes of $t^{\ast ,\kappa }$ we use the method of Gát and
Goginava \cite{gago}. They proved that the\textbf{\ }maximal operator $%
\sigma ^{\alpha ,\ast }$ of $\left( C,\alpha \right) $ $\left( 0<\alpha
<1\right) $ means with respect Walsh-Kaczmarz system is bounded from the
Hardy space $H_{1/\left( 1+\alpha \right) }$ to the space $L_{1/\left(
1+\alpha \right) ,\infty }$. Their proof was depend on the following
inequality%
\begin{equation*}
\left\vert K_{n}^{\alpha ,w}\right\vert \leq \frac{c\left( \alpha \right) }{%
n^{\alpha }}\overset{\left\vert n\right\vert }{\underset{j=0}{\sum }}%
2^{j\alpha }K_{2^{j}}^{w}.
\end{equation*}

Since our estimation of the kernel of $IV$ is the same, it is easy to see
that the proof will be quiet analogously to the Theorem G2.

By using Theorem W we also conclude that the maximal operators $t^{\ast }$
are of weak type-(1,1) and $t_{n}^{\kappa }f\rightarrow f,$ a.e.

Now, we prove the second part of Theorem \ref{Theorem3}. Let $0<p<1/\left(
1+\alpha \right) .$ By combining (\ref{nom3}), (\ref{16a}) and (\ref{17}) we
have that

\begin{eqnarray*}
\frac{\left\Vert t_{2^{n}+1}^{\kappa }f_{n}\right\Vert _{L_{p,\infty }}}{%
\left\Vert f_{n}\right\Vert _{H_{p}}} &\geq &\frac{cq_{0}2^{n\left(
1/p-1\right) }}{Q_{2^{n}+1}}\geq \frac{cq_{0}2^{\alpha n}\left(
2^{n}+1\right) ^{1/p-1-\alpha }}{Q_{2^{n}+1}}\geq c2^{n\left( 1/p-1-\alpha
\right) } \\
&\rightarrow &\infty ,\text{ when }n\rightarrow \infty .
\end{eqnarray*}

Let as prove the third part of Theorem \ref{Theorem3}. By combining (\ref%
{nom2}), (\ref{16a}) and (\ref{17}) we have that%
\begin{equation*}
\frac{\left\Vert t_{2^{n}+1}^{\kappa }f_{n}\right\Vert _{L_{1/\left(
1+\alpha \right) ,\infty }}}{\left\Vert f_{n}\right\Vert _{H_{1/\left(
1+\alpha \right) }}}\geq \frac{cq_{0}2^{n\alpha }}{Q_{2^{n}+1}}\rightarrow
\infty ,\text{ when }n\rightarrow \infty .
\end{equation*}

This complete the proof of Theorem \ref{Theorem3}.
\end{proof}

\textbf{Acknowledgment: }The author would like to thank the referee for
helpful suggestions.

\end{document}